\documentclass[12pt, reqno]{amsart}
\usepackage{amsmath, amsthm, amscd, amsfonts, amssymb, graphicx, color}
\usepackage[bookmarksnumbered, colorlinks, plainpages]{hyperref}
\hypersetup{colorlinks=true,linkcolor=red, anchorcolor=green, citecolor=cyan, urlcolor=red, filecolor=magenta, pdftoolbar=true}
\input{mathrsfs.sty}

\usepackage{xcolor}
\usepackage{ulem}
\newtheorem{theorem}{Theorem}[section]
\newtheorem{lemma}[theorem]{Lemma}
\newtheorem{proposition}[theorem]{Proposition}
\newtheorem{cor}[theorem]{Corollary}
\theoremstyle{definition}

\theoremstyle{remark}
\newtheorem{remark}[theorem]{\bf{Remark}}
\numberwithin{equation}{section}
\begin{document}

\title [Improvement of numerical radius inequalities]{\Small{ Improvement of numerical radius inequalities}}

 \author[P. Bhunia, K. Paul] {Pintu Bhunia,  Kallol Paul}

 \address{(P. Bhunia) Department of Mathematics, Jadavpur University, Kolkata 700 032, West Bengal, India}
 \email{pintubhunia5206@gmail.com}
\email{pbhunia.math.rs@jadavpuruniversity.in} 
\thanks{The first author sincerely acknowledges the financial support received from UGC, Govt. of India in the form of Senior Research Fellowship under the mentorship of Prof Kallol Paul.}

 \address{(K. Paul) Department of Mathematics, Jadavpur University, Kolkata 700 032, West Bengal, India}
 \email{kalloldada@gmail.com}
 \email{kallol.paul@jadavpuruniversity.in}


\renewcommand{\subjclassname}{\textup{2020} Mathematics Subject Classification}\subjclass[]{Primary 47A12, Secondary 15A60, 47A30, 47A50}
\keywords{Numerical radius; Operator norm;  Bounded linear operator; Hilbert space.}

\maketitle

\begin{abstract}
We develop upper and lower bounds for the numerical radius of $2\times 2$ off-diagonal operator matrices, which generalize and improve on the existing ones. We also show  that if $A$ is a bounded linear operator on a complex Hilbert space and $|A|$ stands for the positive square root of $A$, i.e., $|A|=(A^*A)^{1/2}$, then for all $r\geq 1$,
$w^{2r}(A) \leq \frac{1}{4} \big \| |A|^{2r}+|A^*|^{2r} \big \| + \frac{1}{2} \min\left\{ \big \|\Re\big(|A|^r\, |A^*|^r \big) \big \|, w^r(A^2) \right\} $ where $w(A)$, $\|A\|$ and $\Re(A)$, respectively, stand for the numerical radius, the operator norm and the real part of $A$. This (for $r=1$) improves on existing well-known numerical radius inequalities.

\end{abstract}

\section{Introduction}

Let $\mathscr{H}$ be a complex Hilbert space with the inner product $\langle \cdot,\cdot \rangle $ and the corresponding norm $\|\cdot\|$ induced by the inner product. Let $ \mathbb{B}(\mathscr{H})$ denote the $C^*$-algebra of all bounded linear operators on $\mathscr{H}$ with the identity operator  $I$ and the zero operator $O$. Let $A\in \mathbb{B}(\mathscr{H})$. We denote by $|A|=({A^*A})^{\frac{1}{2}}$ the positive square root of $A$, and  $\Re(A)=\frac{1}{2}(A+A^*)$ and $\Im(A)=\frac{1}{2\rm i}(A-A^*)$, respectively, stand for the real and imaginary part of $A$. The numerical range of $A$, denoted by $W(A)$, is defined as $W(A)=\left \{\langle Ax,x \rangle: x\in \mathscr{H}, \|x\|=1 \right \}.$
We denote by $\|A\|$ and $w(A)$ the operator norm and the numerical radius of $A$, respectively, and are defined as  $$\|A\|=\sup  \left \{\| Ax\|: x\in \mathscr{H}, \|x\|=1 \right \}$$ and $$w(A)=\sup \left \{|\langle Ax,x \rangle|: x\in \mathscr{H}, \|x\|=1 \right \}.$$ 
It is well-known that the numerical radius $ w(\cdot)$ defines a norm on $\mathbb{B}(\mathscr{H})$ and is equivalent to the operator norm $\|\cdot\|$. In fact, the following double inequality holds:
\begin{eqnarray}\label{eqv}
\frac{1}{2} \|A\|\leq w({A})\leq\|A\|.
\end{eqnarray}
The inequalities in (\ref{eqv}) are sharp. The first inequality becomes equality if $A^2=O$, and the second one turns into equality if $A$ is normal. For various refinements of (\ref{eqv}), we refer the reader to \cite{BBP_aofa, BP_adm, P18, PK_rm, PSK, YAM1} and references therein. In particular, Kittaneh \cite{K05} improved the inequalities in (\ref{eqv}) by establishing that 
\begin{eqnarray}\label{k5}
\frac{1}{4}\|A^*A+A{A}^*\|\leq w^2({A})\leq\frac{1}{2}\|A^*A+A{A}^*\|.
\label{d}\end{eqnarray}
\smallskip
 Kittaneh \cite{K03} also improved the upper bound of $w(T)$ in (\ref{eqv}) to show that 
\begin{eqnarray}\label{k3}
 w({A})\leq\frac{1}{2}\left( \|A\|+\|A^2\|^{1/2} \right ).
\end{eqnarray}
Further, Abu-Omar and Kittaneh \cite{AK15} obtained the following inequality which refines both the upper bounds in  (\ref{k5})  and  (\ref{k3}):
\begin{eqnarray}\label{k15}
 w^2({A})\leq\frac{1}{4}\|A^*A+A{A}^*\|+\frac{1}{2}w(A^2).
\end{eqnarray}
Recently, Bhunia and Paul \cite{PK} also improved both the upper bounds  in (\ref{k5})  and  (\ref{k3}) by developing that
 \begin{eqnarray}\label{pk21}
 w^2({A})\leq\frac{1}{4}\|A^*A+A{A}^*\|+\frac{1}{2}w(|A||A^*|).
 \end{eqnarray}
 
 In this paper,  we derive  inequalities for the bounds of  the numerical radius which generalize and improve on both in (\ref{k15})  and  (\ref{pk21}). Further, we obtain a lower bound for the numerical radius which generalizes and improves on the existing ones. Applications of the obtained inequalities are also given.

\section{Main Results}

We begin our work with noting that for $A,B\in \mathbb{B}(\mathscr{H})$, the $2\times 2$ off-diagonal operator matrix $\begin{pmatrix}
O&A\\ B&O 
\end{pmatrix}\in \mathbb{B}(\mathscr{H}\oplus \mathscr{H})$  and the numerical radius of the matrix is denoted by $w\begin{pmatrix}
O&A\\ B&O 
\end{pmatrix}$. To achieve our aim in this paper we need the following four lemmas. First lemma follows easily.

\begin{lemma}\label{lem1}
	If $A,B\in \mathbb{B}(\mathscr{H})$, then
	$$w\begin{pmatrix}
	A&O\\ O&B 
	\end{pmatrix}=\max \Big \{w(A), w(B)  \Big\}.$$
\end{lemma}

Second lemma deals with positive operators, and  is known as McCarthy inequality.

\begin{lemma}$($\cite[p. 20]{simon}$)$.\label{lem2}
	If $A\in \mathbb{B}(\mathscr{H})$ is positive, then for $r\geq 1$ 
	\[\langle Ax,x\rangle^r\leq \langle A^rx,x\rangle,\]  
	for all $x \in \mathscr{H}$ with $\|x\|=1.$
\end{lemma}

Third lemma deals with  vectors in $\mathscr{H}$, and is known as Buzano's inequality.

\begin{lemma}$($\cite{BU}$)$\label{lem3}
	If $x,y,e\in \mathscr{H}$ with $\|e\|=1,$ then 
	\[|\langle x,e\rangle \langle e,y\rangle|\leq \frac{1}{2}\Big(\|x\| \|y\|+|\langle x,y\rangle|\Big).\]
\end{lemma}

Fourth lemma is known as  mixed Schwarz inequality.

\begin{lemma}$($\cite{H}$)$\label{lem4}
	If $A\in \mathbb{B}(\mathscr{H})$, then
	\[|\langle Ax,y \rangle| \leq \langle |A|x,x \rangle ^{1/2} \langle |A^*|y,y \rangle ^{1/2},\]
	for all $x,y \in \mathscr{H}$.
\end{lemma}

We are now in a position to prove our first desired inequality.

\begin{theorem}\label{th1}
	If $A,B\in \mathbb{B}(\mathscr{H})$, then 
	\begin{eqnarray*}
&&	w^{2r}\begin{pmatrix}
		O&A\\ B&O
	\end{pmatrix} \\
&& \leq  \frac{1}{4} \max \Big \{ \left \| |B|^{2r}+|A^*|^{2r}  \right \|, \left \| |A|^{2r}+|B^*|^{2r} \right \| \Big\}+\frac{(1-\alpha)}{2} \max \Big \{w^r(AB), w^r(BA) \Big \}\\
&&  + \frac{\alpha}{2} \max \Big \{\left \|\Re(|B|^r|A^*|^r) \right \|,  \left \|\Re(|A|^r|B^*|^r) \right \|\Big \},
	\end{eqnarray*}
for all $\alpha\in [0,1]$ and for all $r\geq 1$.
\end{theorem}

\begin{proof}
Let $T=\begin{pmatrix}
O&A\\ B&O
\end{pmatrix}$	and let $x\in \mathscr{H}\oplus \mathscr{H} $ with $\|x\|=1$. Then we have,
\begin{eqnarray*}
&& |\langle Tx,x\rangle|^{2r}\\
&=&\alpha |\langle Tx,x\rangle|^{2r}+(1-\alpha)|\langle Tx,x\rangle|^{2r}\\
&\leq& \alpha \left (\langle |T|x,x\rangle^{1/2} \,\langle |T^*|x,x\rangle^{1/2} \right)^{2r}+ (1-\alpha)|\langle Tx,x\rangle|^{2r}\,\,\,\Big(\textit{by Lemma \ref{lem4}}\Big)\\
&=& \alpha \left (\langle |T|x,x\rangle^{r/2} \,\langle |T^*|x,x\rangle^{r/2} \right)^{2}+ (1-\alpha)|\langle Tx,x\rangle|^{2r}\\
&\leq& \alpha \left (\langle |T|^rx,x\rangle^{1/2} \,\langle |T^*|^rx,x\rangle^{1/2} \right)^{2}+ (1-\alpha)|\langle Tx,x\rangle|^{2r} \,\,\,\Big(\textit{by Lemma \ref{lem2}}\Big)\\
&\leq & \alpha \left ( \frac{\langle |T|^r x,x\rangle+\langle |T^*|^r x,x\rangle}{2} \right)^2+(1-\alpha)|\langle Tx,x\rangle|^{2r} \\
&= & \alpha \left ( \frac{\langle ( |T|^r+|T^*|^r) x,x\rangle}{2} \right)^2+(1-\alpha)|\langle Tx,x\rangle|^{2r}  \\
&\leq & \frac{\alpha}{4} \langle ( |T|^r+|T^*|^r)^2 x,x\rangle  +(1-\alpha)|\langle Tx,x\rangle\, \langle x,T^*x\rangle|^r \,\,\,\Big(\textit{by Lemma \ref{lem2}}\Big)\\
&\leq & \frac{\alpha}{4} \langle ( |T|^r+|T^*|^r)^2 x,x\rangle  +\frac{(1-\alpha)}{2^r}\big(\|Tx\| \, \|T^*x\|+ |\langle Tx,T^*x\rangle|\big)^r \,\,\,\Big(\textit{by Lemma \ref{lem3}}\Big)\\
&\leq & \frac{\alpha}{4} \langle ( |T|^r+|T^*|^r)^2 x,x\rangle  +\frac{(1-\alpha)}{2}\big(\|Tx\|^r \, \|T^*x\|^r+ |\langle Tx,T^*x\rangle|^r\big)\\
&\leq & \frac{\alpha}{4} \langle ( |T|^r+|T^*|^r)^2 x,x\rangle  +\frac{(1-\alpha)}{2}\left(\frac{\|Tx\|^{2r}+ \|T^*x\|^{2r}}{2}+ |\langle T^2x,x\rangle|^r\right)\\
&\leq & \frac{\alpha}{4} \langle ( |T|^r+|T^*|^r)^2 x,x\rangle  +\frac{(1-\alpha)}{2}\left(\frac{\langle (|T|^{2r}+ |T^*|^{2r})x,x \rangle}{2}+ |\langle T^2x,x\rangle|^r\right)\\
&= & \frac{\alpha}{4} \langle \left( |T|^{2r}+|T^*|^{2r}+|T|^r\,|T^*|^r+|T^*|^r\,|T|^r \right) x,x\rangle  \\
&& \,\,\,\,\,\,\,\,\,\,\,\,\,\,\,\,\,\,\,\,\,\,\,\,\,\,\,\,\,\,\,\,\,\,\,\,\,\,\,\,\,\,\,\,\,\,\,\,\,\,\,\,\,\,\,\,\,\,\,\,\,\,\,\,\,\,\,\,\,\,\,\,\,\,\,\,\,\,\,\,+\frac{(1-\alpha)}{2}\left(\frac{\langle (|T|^{2r}+ |T^*|^{2r})x,x \rangle}{2}+ |\langle T^2x,x\rangle|^r\right)\\
&= & \frac{1}{4} \langle \left( |T|^{2r}+|T^*|^{2r} \right) x,x\rangle +\frac{\alpha}{2} \langle \Re(|T|^r\,|T^*|^r)x,x\rangle +\frac{(1-\alpha)}{2}|\langle T^2x,x\rangle|^r\\
&\leq & \frac{1}{4} \langle \left( |T|^{2r}+|T^*|^{2r} \right) x,x\rangle +\frac{\alpha}{2} |\langle \Re(|T|^r\,|T^*|^r)x,x\rangle| +\frac{(1-\alpha)}{2}|\langle T^2x,x\rangle|^r\\
&\leq & \frac{1}{4} \left \|  |T|^{2r}+|T^*|^{2r} \right \| +\frac{\alpha}{2} \left \| \Re(|T|^r\,|T^*|^r) \right \|+\frac{(1-\alpha)}{2} w^r( T^2 ).
\end{eqnarray*}
Thus, taking supremum over $\|x \|=1$, we have
\begin{eqnarray}
w^{2r}(T)
&\leq & \frac{1}{4} \left \|  |T|^{2r}+|T^*|^{2r} \right \| +\frac{\alpha}{2} \left \| \Re(|T|^r\,|T^*|^r) \right \|+\frac{(1-\alpha)}{2} w^r( T^2 ).
\end{eqnarray}
That is,
\begin{eqnarray*}
w^{2r}\begin{pmatrix}
	O&A\\ B&O
\end{pmatrix}
&\leq & \frac{1}{4} \left \|  \begin{pmatrix}
	|B|^{2r}+|A^*|^{2r} &O\\ O&|A|^{2r}+|B^*|^{2r}
\end{pmatrix} \right \| \\
&& \,+\frac{\alpha}{2} \left \|   \begin{pmatrix}
\Re(|B|^r\,|A^*|^r)&O\\ O&\Re(|A|^r\,|B^*|^r)
\end{pmatrix} \right \|+\frac{(1-\alpha)}{2} w^r\begin{pmatrix}
AB&O\\ O&BA
\end{pmatrix}.
\end{eqnarray*}
Therefore, the required inequality follows from Lemma \ref{lem1}.

\end{proof}

\begin{remark}

In particular, considering $\alpha=1$ and $r=1$ in the above theorem we get that
\begin{eqnarray*}
	&& w^2\begin{pmatrix}
		O&A\\ B&O
	\end{pmatrix} \leq \\
& & \frac{1}{4} \max \Big \{ \left \| |B|^2+|A^*|^2  \right \|, \left \| |A|^2+|B^*|^2  \right \| \Big\}  + \frac{1}{2} \max \Big \{\left \|\Re(|B||A^*|) \right \|,  \left \|\Re(|A||B^*|) \right \|\Big \},
\end{eqnarray*}
which refines the existing one \cite[Th. 2.10]{PK2}, namely
\begin{eqnarray*}
&&	w^2\begin{pmatrix}
		O&A\\ B&O
	\end{pmatrix} \leq \\
& & \frac{1}{4} \max \Big \{ \left \| |B|^2+|A^*|^2  \right \|, \left \| |A|^2+|B^*|^2  \right \| \Big\} + \frac{1}{2} \max \Big \{w(|B||A^*|),  w(|A||B^*|) \Big \}.
\end{eqnarray*}

\end{remark}

To obtain our next result we need the following lemma, which can be found in  \cite[Lemma 2.1]{HKS}.

\begin{lemma}\label{lem5}
	If $A,B\in \mathbb{B}(\mathscr{H})$, then
	$$w \left(\begin{array}{cc}
	A & B\\
	B& A
	\end{array}\right)=\max \Big \{w(A+B),w(A-B) \Big\}.$$
In particular, \,\,\,\,\,\,\, \,\,\,\,\,\,$w \left( \begin{array}{cc}
	O & B\\
	B & O
	\end{array} \right)=w(B).$
\end{lemma}

\begin{cor}\label{cor1}
If $A\in \mathbb{B}(\mathscr{H})$, then
\begin{eqnarray*}
w^{2r}(A) \leq \frac{1}{4} \left \| |A|^{2r}+|A^*|^{2r}  \right \|+ \frac{1}{2} \min \Big \{ \left \|\Re(|A|^r|A^*|^r) \right \|, \,\,  w^r(A^2) \Big \},
\end{eqnarray*}
for all $r\geq 1$.	
\end{cor}
\begin{proof}
	Considering $A=B$ in Theorem \ref{th1}, and then using Lemma \ref{lem5} we infer that
	\begin{eqnarray}\label{bound1}
		w^{2r}(A) \leq \frac{1}{4} \left \| |A|^{2r}+|A^*|^{2r} \right \|+ \frac{1}{2} \Big \{ \alpha \left \|\Re(|A|^r|A^*|^r) \right \| +(1-\alpha)  w^r(A^2) \Big \},
	\end{eqnarray}
for all $\alpha \in [0,1]$ and for all $r\geq 1$. This implies the desired inequality. 
\end{proof}

\begin{remark}
	(i) Since $\left \|\Re(|A||A^*|) \right \|\leq w(|A||A^*|) $, so we would like to remark that Corollary \ref{cor1} (for $r=1$) gives stronger inequality than that in both (\ref{k15}) and (\ref{pk21}). \\
	(ii) If for norm one sequences $\{x_n\}$ in $\mathscr {H}$ with $|\langle \Re(|T|~|T^*|)  x_n,x_n \rangle| \to \|\Re(|T|~|T^*|)\|$ and $|\langle \Im(|T|~|T^*|)  x_n,x_n \rangle| \to \lambda ( \neq 0 )$, then Corollary \ref{cor1} (for $r=1$) gives strictly stronger inequality than that in (\ref{pk21}). 
\end{remark}

We next prove the following theorem.

\begin{theorem}
	Let $A\in \mathbb{B}(\mathscr{H})$  be such that $\Re\big(|A||A^*| \big)=O.$ Then $\overline{W(A)}$ is a circular disk with center at the origin and radius $\frac{1}{2}\sqrt{\|A^*A+AA^*\|}$.  
\end{theorem}

\begin{proof}
	From Corollary \ref{cor1}, we get for the case $r=1,$
	\begin{eqnarray}\label{main}
	w^{2}(A) \leq \frac{1}{4} \left \| |A|^{2}+|A^*|^{2} \right \|+ \frac{1}{2}   \left \|\Re(|A||A^*|) \right \|. 
	\end{eqnarray}
The first inequality in (\ref{k5}) together with (\ref{main}) we infer that
 $  w(A) = \frac{1}{2} \sqrt{\left \| |A|^{2}+|A^*|^{2}  \right \|} = \frac{1}{2} \sqrt{\|A^*A+AA^*\|} $ 	if  $\Re\big(|A||A^*| \big)=O.$  
	Therefore, from \cite[Th. 2.14]{BPaul} it follows that  $\overline{W(A)}$ is a circular disk with center at the origin and radius $\frac{1}{2}\sqrt{\|A^*A+AA^*\|}$.  
	
\end{proof}


We observe that the converse of the above result is, in general, not true. There exists  an operator $A\in \mathbb{B}(\mathscr{H})$ for which  $w(A) = \frac{1}{2}\sqrt{  \left \| |A|^2+|A^*|^2  \right \|},$ but $\Re\big(|A||A^*| \big)\neq O$.
	Consider  $A=\begin{pmatrix}
		0&1&0\\
		0&0&1\\
		0&0&0
	\end{pmatrix}$. Then one can verify  that $w(A) = \frac{1}{2}\sqrt{  \left \| |A|^2+|A^*|^2  \right \|}  =\frac{1}{\sqrt{2}} $, but  $\Re\big(|A||A^*| \big)=\begin{pmatrix}
		0&0&0\\
		0&1&0\\
		0&0&0
	\end{pmatrix}$.


Next we obtain an inequality for the numerical radius of the sum of $n$ operators which generalizes Theorem \ref{th1}. For this, we need the following 	Bohr's inequality which deals with positive real numbers.

\begin{lemma}$($\cite{v}$)$\label{lem4v}
	If $a_i\geq 0$ for each $i=1,2,\ldots,n$, then 
	\[\left( \sum_{i=1}^na_i\right)^r \leq n^{r-1}\sum_{i=1}^na_i^r,\] for all $r\geq 1$.
\end{lemma}

\begin{theorem}\label{th1gen}
	If $A_i \in \mathbb{B}(\mathscr{H})$ for $ \ i=1,2, \ldots,n$,  then 
	\begin{eqnarray*}
		w^{2r}\left( \sum_{i=1}^{n}A_i \right) &\leq&     \frac{n^{2r-1}}{4} \left \| \sum_{i=1}^{n}  \left( |A_i|^{2r}+ |A_i^*|^{2r} \right)      \right\| \\        && +\frac{n^{2r-1}}{2} \left(\alpha \left \| \sum_{i=1}^{n}\Re \left (|A_i|^r|A_i^*|^r \right)\right \|+(1-\alpha) \sum_{i=1}^{n}w^r \left( A_i^{2} \right)      \right),  
	\end{eqnarray*}
	for all $\alpha\in [0,1]$ and for all $r\geq 1.$
\end{theorem}

\begin{proof}
	Let $x\in \mathscr{H}$ with $\|x\|=1.$ Then  by using Lemma \ref{lem4v} we infer that 
	\begin{eqnarray*}
		\left |\left \langle \left( \sum_{i=1}^{n}A_i \right)x,x \right \rangle \right |^{2r} &=& \left |\sum_{i=1}^{n} \left \langle  A_i x,x \right \rangle \right |^{2r}\\
		&\leq&  \left (\sum_{i=1}^{n} |\left \langle  A_i x,x \right \rangle| \right )^{2r}\\
		&\leq&  n^{2r-1} \left (\sum_{i=1}^{n} |\left \langle  A_i x,x \right \rangle|^{2r} \right )\\
		&=&  n^{2r-1} \left (   \alpha \sum_{i=1}^{n} |\left \langle  A_i x,x \right \rangle|^{2r} +(1-\alpha) \sum_{i=1}^{n} |\left \langle  A_i x,x \right \rangle|^{2r} \right ).
	\end{eqnarray*}
	Now proceeding similarly  as in Theorem \ref{th1} we have the desired inequality.
\end{proof}

Next we obtain a lower bound for the numerical radius of the $2 \times 2$ off-diagonal operator matrix $\begin{pmatrix} 
O&A\\ B&O
\end{pmatrix}$. By considering the unitary operator matrix $\begin{pmatrix} 
O&I\\ I&O
\end{pmatrix}$, the weak unitary invariance property for the numerical radius gives that $w\begin{pmatrix} 
O&A\\ B&O
\end{pmatrix}=w\begin{pmatrix} 
	O&B\\ A&O
\end{pmatrix}$.

\begin{theorem}\label{th2}
 If $A,B\in \mathbb{B}(\mathscr {H})$, then
 \begin{eqnarray}\label{lower1}
  w\begin{pmatrix} 
  	O&A\\ B&O
  \end{pmatrix}\geq \frac{1}{2} \Big \| \Re(A)+{\rm i}\, \Im(B) \Big \|+\frac{1}{4} \Big | \, \|A+B^*\|-\|A-B^*\|\,  \Big |,
\end{eqnarray}
and 
 \begin{eqnarray}\label{lower2}
w\begin{pmatrix} 
O&A\\ B&O
\end{pmatrix}\geq \frac{1}{2} \Big \| \Re(B)+{\rm i}\, \Im(A) \Big \|+\frac{1}{4} \Big | \, \|A+B^*\|-\|A-B^*\|\,  \Big |.
\end{eqnarray}
 \end{theorem}
\begin{proof}
It is well-known that 	$$w\begin{pmatrix} 
O&A\\ B&O
\end{pmatrix}\geq \left\| \Re\begin{pmatrix} 
O&A\\ B&O
\end{pmatrix} \right\|=\frac{1}{2}\| A+B^*\|$$ and $$w\begin{pmatrix} 
O&A\\ B&O
\end{pmatrix}\geq \left\| \Im\begin{pmatrix} 
O&A\\ B&O
\end{pmatrix} \right\|=\frac{1}{2}\| A-B^*\|.$$
Therefore, we have
\begin{eqnarray*}
w\begin{pmatrix} 
	O&A\\ B&O
\end{pmatrix} &\geq & \frac{1}{2} \max \Big \{ \|A+B^*\|, \|A-B^*\|\Big\}\\
&=& \frac{1}{2} \left ( \frac{\|A+B^*\|+ \|A-B^*\| }{2}+ \frac{|\,\|A+B^*\|-\|A-B^*\|\,| }{2} \right) \\
&=& \frac{1}{2} \left ( \frac{\|A+B^*\|+ \|A^*-B\| }{2}+ \frac{|\,\|A+B^*\|-\|A-B^*\|\,| }{2} \right) \\
&\geq & \frac{1}{2} \left ( \frac{\|(A+B^*)+(A^*-B)\| }{2}+ \frac{|\,\|A+B^*\|-\|A-B^*\|\,| }{2} \right) \\
&=& \frac{1}{2} \left ( \|\Re(A)-{\rm i}\,\Im(B) \|+ \frac{|\,\|A+B^*\|-\|A-B^*\|\,| }{2} \right) \\
&=& \frac{1}{2} \left \|\Re(A)-{\rm i}\,\Im(B) \right \|+ \frac{1 }{4}\Big|\,\|A+B^*\|-\|A-B^*\|\,\Big|.
\end{eqnarray*}
This implies the inequality (\ref{lower1}). Interchanging $A$ and $B$ in (\ref{lower1}) we get the inequality (\ref{lower2}). 
\end{proof}
As a consequence of Theorem \ref{th2} we get the following corollaries.
\begin{cor}\label{lowerbound1}
	Let $ A \in \mathbb{B}(\mathscr{H}).$  Then 
\begin{eqnarray}\label{lowerLAA}
	w(A)\geq \frac{1}{2}\big\|A \big\|+\frac{1}{2} \Big|\, \|\Re(A)\|-\|\Im(A)\|\,\Big|.
\end{eqnarray}
\end{cor}
\begin{proof} This follows clearly from Theorem \ref{th2}  by considering $A=B.$ 
\end{proof}

\begin{cor}\label{cor_applLAA}
	 Let $A,B\in \mathbb{B}(\mathscr {H}).$ Then 
	\begin{eqnarray}\label{applLAA}
		w\begin{pmatrix} 
			O&A\\ B&O
		\end{pmatrix}\geq \frac{1}{2} \max \Big \{\|A\|, \|B\|  \Big \}+\frac{1}{4} \Big | \, \|A+B^*\|-\|A-B^*\|\,  \Big |.
	\end{eqnarray}
\end{cor}
\begin{proof} Considering the operator $ \begin{pmatrix} 
		O&A\\ B&O
	\end{pmatrix} $ and applying the inequality (\ref{lowerLAA}) we get the desired inequality (\ref{applLAA}).
\end{proof}
\begin{remark}\label{rem1}
(i) Consider the matrix  $A=\begin{pmatrix} 
3&0\\ 0&0
\end{pmatrix}$ and $B=\begin{pmatrix} 
2+3{\rm i}&0\\ 0&0
\end{pmatrix}$. Then, $\max \Big \{\|A\|, \|B\|  \Big \}=\sqrt{13}$ and $\left \|\Re(A)+{\rm i}\,\Im(B) \right \|=\sqrt{18}$. Again, if we assume that  $A=\begin{pmatrix} 
3&0\\ 0&0
\end{pmatrix}$ and $B=\begin{pmatrix} 
0&0\\ 0&2+3{\rm i}
\end{pmatrix}$, then $\max \Big \{\|A\|, \|B\|  \Big \}=\sqrt{13}$ and $\left \|\Re(A)+{\rm i}\,\Im(B) \right \|=\sqrt{9}$. Thus, we would like to remark that the inequalities obtained in Theorem \ref{th2} and Corollary \ref{cor_applLAA} are, in general, not comparable.\\
(ii)  We observe that the inequalities in Theorem \ref{th2} and Corollary \ref{cor_applLAA} are sharp.
\end{remark}

Next, we study the following necessary conditions for the equalities of $w\begin{pmatrix} 
O&A\\ B&O
\end{pmatrix} = \frac{1}{2} \left \|\Re(A)+{\rm i}\,\Im(B) \right \|$ and $w\begin{pmatrix} 
O&A\\ B&O
\end{pmatrix} = \frac{1}{2} \left \|\Re(B)+{\rm i}\,\Im(A) \right \|$.

\begin{proposition}\label{prop1}
	Let $A,B\in \mathbb{B}(\mathscr{H})$. Then the following results hold.\\
(i)	If $w\begin{pmatrix} 
	O&A\\ B&O
	\end{pmatrix} = \frac{1}{2} \left \|\Re(A)+{\rm i}\,\Im(B) \right \|,$ then
	$\|A+B^*\|=\|A-B^*\|= \left \|\Re(A)+{\rm i}\,\Im(B) \right \|$.\\
(ii) If $w\begin{pmatrix} 
	O&A\\ B&O
\end{pmatrix} = \frac{1}{2} \left \|\Re(B)+{\rm i}\,\Im(A) \right \|$, then 
$\|A+B^*\|=\|A-B^*\|= \left \|\Re(B)+{\rm i}\,\Im(A) \right \|$. 

\end{proposition}

\begin{proof}
  Suppose $w\begin{pmatrix} 
	O&A\\ B&O
	\end{pmatrix} = \frac{1}{2} \left \|\Re(A)+{\rm i}\,\Im(B) \right \|.$  Then from the inequality (\ref{lower1}) it follows that 
	$\|A+B^*\|=\|A-B^*\|$. Therefore,
	\begin{eqnarray*}
	\frac{1}{2}\|A+B^*\|&\leq& w\begin{pmatrix} 
		O&A\\ B&O
	\end{pmatrix} \\
& = & \frac{1}{2} \left \|\Re(A)+{\rm i}\,\Im(B) \right \| \\
&=& \frac{1}{2} \left \| \frac{A+A^*}{2}+{\rm i}\,\frac{B-B^*}{2{\rm i}} \right \| \\
&\leq&  \frac{1}{2}  \left (  \frac{1}{2} \|A-B^*\|+ \frac{1}{2} \|A+B^*\| \right )\\
&=&  \frac{1}{2} \|A+B^*\| \,\,\,\,\,\,\Big( \textit{since}\,\,\|A+B^*\|=\|A-B^*\|\Big).
	\end{eqnarray*}
This completes the proof of (i). 
The proof of (ii) follows from (i) by interchanging $A$ and $B$.
\end{proof}

Finally, as a consequence of Theorem \ref{th2} we obtain the following inequality.

\begin{cor}\label{cor2}
	If $A,B\in \mathbb{B}(\mathscr{H})$, then 
	$$ w\begin{pmatrix} 
	O&A\\ B&O
	\end{pmatrix} \geq \frac{1}{4}\|A+B\|+\frac{1}{4}|a-b|+\frac{1}{4} \Big| \|A+B^*\|-\|A-B^*\|\Big|,$$
	where $a= \left \|\Re(A)+{\rm i}\,\Im(B) \right \|,b= \left \|\Re(B)+{\rm i}\,\Im(A) \right \|.$
	This inequality is sharp.
\end{cor}
\begin{proof}
	It follows from (\ref{lower1}) and (\ref{lower2}) that 
	\begin{eqnarray}\label{lower3}
	w\begin{pmatrix} 
	O&A\\ B&O
	\end{pmatrix} \geq \frac{1}{2} \max \{a,b\}+ \frac{1}{4} \Big| \|A+B^*\|-\|A-B^*\|\Big|.
	\end{eqnarray}
	Now, 
\begin{eqnarray}\label{lower4}
	\max \{a,b\} & = & \frac{1}{2}(a+b)+\frac{1}{2}|a-b|\geq \frac{1}{2}\|A+B\|+\frac{1}{2}|a-b|.
\end{eqnarray}
Therefore, combining (\ref{lower4}) with (\ref{lower3}) we infer that the desired inequality. The proof of sharpness is trivial.
\end{proof}

We end this section with the following result.
\begin{proposition}
	Let $A,B\in \mathbb{B}(\mathscr{H})$. If $w\begin{pmatrix} 
	O&A\\ B&O
	\end{pmatrix} = \frac{1}{4}\|A+B\|$, then the following results hold:
	
	(i) $\|A\|=\|B\|$.
	
	(ii) $\|A+B^*\|=\|A-B^*\|$.
	
	(iii) $\left \|\Re(A)+{\rm i}\,\Im(B) \right \|= \left \|\Re(B)+{\rm i}\,\Im(A) \right \|.$
	
\end{proposition}

\begin{proof}
	The proof of (i) follows from (\ref{applLAA}). The proofs of (ii) and (iii) follow from Corollary \ref{cor2}.
\end{proof}

\bibliographystyle{amsplain}

\begin{thebibliography}{99}
	
\bibitem{AK15}	 A. Abu-Omar and F. Kittaneh, Upper and lower bounds for the numerical radius with an application to involution operators,
\textit{ Rocky Mountain J. Math.} \textbf{ 45 } (2015), no. 4, 1055--1065.




\bibitem{PK} P. Bhunia and K. Paul, New upper bounds for the numerical radius of Hilbert space operators, \textit{ Bull. Sci. Math.} \textbf{ 167 } (2021), Paper No. 102959, 11 pp. 




\bibitem{BBP_aofa} P. Bhunia, S. Bag, and K. Paul, Bounds for zeros of a polynomial using numerical radius of Hilbert space operators, \textit{Ann. Funct. Anal.} \textbf{12} (2021), no. 2, Paper No. 21, 14 pp.

\bibitem{BP_adm} P. Bhunia and K. Paul, Furtherance of numerical radius inequalities of Hilbert space operators,  \textit{Arch. Math. (Basel)} (2021), https://doi.org/10.1007/s00013-021-01641-w



\bibitem{P18} P. Bhunia and K. Paul, Refinements of norm and numerical radius inequalities, \textit{Rocky Mountain J. Math.} (2021), in press. 


\bibitem{PK_rm} P. Bhunia and K. Paul,  Proper improvement of well-known numerical radius inequalities and their applications, \textit{ Results Math.} \textbf{ 76} (2021), no. 4, Paper No. 177.




\bibitem{BPaul} P. Bhunia and K. Paul, Development of inequalities and characterization of equality conditions for the numerical radius, \textit{Linear Algebra Appl.} \textbf{630} (2021), 306--315.

\bibitem{PK2} P. Bhunia and K. Paul, Numerical radius inequalities of $2 \times 2$ operator matrices, preprint, https://arxiv.org/abs/2105.09718

\bibitem{PSK} P. Bhunia, S. Bag and K. Paul, Numerical radius inequalities and its applications in estimation of zeros of polynomials, \textit{Linear Algebra Appl.} \textbf{ 573} (2019), 166--177.

\bibitem{BU} M. L. Buzano, Generalizzazione della diseguaglianza di Cauchy--Schwarz (Italian), \textit{Rend Sem Mat Univ E Politech Torino.} 1974 \textbf{31} (1971/73), 405--409.






\bibitem{HKS} O. Hirzallah, F. Kittaneh, K. Shebrawi, Numerical radius inequalities for certain $2\times 2$ operator matrices, \textit{Integral Equ. Oper. Theory} \textbf{ 71} (2011), 129--147.


\bibitem{H} P. R. Halmos, A Hilbert space problems book, Springer Verlag, New York, 1982.

\bibitem{K05} F. Kittaneh, Numerical radius inequalities for Hilbert space operators, \textit{Studia Math.} \textbf{168} (2005), no. 1, 73--80.


\bibitem{K03} F. Kittaneh, A numerical radius inequality and an estimate for the numerical radius of the Frobenius companion matrix, \textit{ Studia Math.} \textbf{ 158} (2003), no. 1, 11--17.


 


\bibitem{simon} B. Simon, Trace ideals and their applications, Cambridge University Press, 1979.

\bibitem{v} M. P. Vasi\'c,  D. J. Ke\^cki\'c, Some inequalities for complex numbers, \textit{Math. Balkanica} \textbf{1} (1971), 282--286.

\bibitem{YAM1} T. Yamazaki, On upper and lower bounds for the numerical radius and an equality condition, \textit{Studia Math.} \textbf{178} (2007), no. 1, 83--89. 

\end{thebibliography}

\end{document}